\documentclass[a4paper]{amsart}

\usepackage{graphicx}
\usepackage{amsmath,amssymb,enumerate,amsbsy}
\usepackage[all]{xy}
\usepackage[active]{srcltx}
\usepackage{color}
\usepackage{cite}

 \newtheorem{thm}{Theorem}[section]
 
 \newtheorem{lem}[thm]{Lemma}
 \newtheorem{prop}[thm]{Proposition}
 \theoremstyle{definition}
 
 \theoremstyle{remark}

 \numberwithin{equation}{section}
 
\newcommand{\nc}[2]{\newcommand{#1}{#2}}
\newcommand{\rnc}[2]{\renewcommand{#1}{#2}}
\rnc{\[}{\begin{equation}}
\rnc{\]}{\end{equation}}
\nc{\wegengruen}{\end{equation}}
 
\newcommand{\Z}{\mathbb{Z}}
\newcommand{\N}{\mathbb{N}}
\newcommand{\R}{\mathbb{R}}
\newcommand{\C}{\mathbb{C}}
 
 \newcommand{\hsp}{{\hspace{-1pt}}}
\newcommand{\hs}{{\hspace{1pt}}}

\newcommand{\cD}{\mathcal{D}}
\newcommand{\hH}{\mathcal{H}}    
\newcommand{\K}{{\mathcal{K}}}

\newcommand{\cO}{\mathcal{O}}

\newcommand{\T}{{\mathcal{T}}}

\newcommand{\A}{{\mathcal{A}}}

\newcommand{\F}{{\mathcal{F}}}

\newcommand{\oqd}{\mathbb{D}_q}

\newcommand{\Uq}{\cO(\oqd)}  
\newcommand{\Fq}{\F(\oqd)} 
\newcommand{\dFq}{\F^{(1)}(\oqd)} 
\newcommand{\dd}{\mathrm{d}}

\newcommand{\im}{\mathrm{i}}
\newcommand{\E}{\mathrm{e}}

\newcommand{\id}{{\mathrm{id}}}

\newcommand{\Tr}{\mathrm{Tr}}
\newcommand{\spec}{\mathrm{spec}}
\newcommand{\dom}{\mathrm{dom}}

\newcommand{\lN}{\ell_2(\N)}
\newcommand{\lZ}{\ell_2(\Z)}
\newcommand{\s}{\sigma}
\renewcommand{\a}{\alpha}

\newcommand{\half}{\mbox{$\frac{1}{2}$}}
\newcommand{\dz}{\mbox{$\frac{\partial}{\partial z}$}}
\newcommand{\dbz}{\mbox{$\frac{\partial}{\partial \bar z}$}}
\newcommand{\dt}{\mbox{$\frac{\partial}{\partial t}$}}

\newcommand{\ip}[2]{\langle{#1},{#2}\rangle} 

\newcommand{\ra}{\rightarrow}
\newcommand{\lra}{\longrightarrow}

\newcommand{\ot}{\otimes}

\newcommand{\op}{\mathrm{op}}
\def\trho{\tilde\rho} 

    \def\Aq{\cO(\mathrm{SU}_q(2))} 
\def\SU{\mathrm{SU}(2)}   
   
\def\CSUq{C(\mathrm{SU}_q(2))} 
\def\su{\mathrm{su}(2)}
\def\pr{\mathrm{pr}}
\renewcommand{\S}{\mathbb{S}}

\newcommand{\disc}{\mathbb{D}}
\newcommand{\qd}{\bar{\mathbb{D}}_q}
\newcommand{\xx}{\mathbf{x}}

\title{Twisted Dirac operator on quantum SU(2) in disc coordinates}

\author{Ulrich Kr\"ahmer}

\address{Technical University Dresden \\
Institute of Geometry \\
01062 Dresden \\
Germany}

\email{ulrich.kraehmer@tu-dresden.de}

\author{Elmar Wagner}

\address{Instituto de F\'isica y Matem\'aticas\\
Universidad Michoacana de San Nicol\'as de Hidalgo\\ 
58040 Morelia\\ M\'exico} 

\email{elmar@ifm.umich.mx} 



\dedicatory{Dedicated to Professor Nikolai Vasilevski on occasion of his 70th birthday.}

\begin{document}

\begin{abstract}
The quantum disc is used to define 
a noncommutative analogue of a  
dense coordinate chart and of
left-invariant vector fields on quantum SU(2).
This yields two twisted Dirac operators 
for different twists  
that are related by a gauge
transformaton and have bounded twisted
commutators with a suitable algebra of 
differentiable functions on quantum
SU(2). 
\end{abstract}

\maketitle

\section{Introduction} 
Connes' noncommutative differential geometry 
\cite{C-buch,C-ihes} provides in particular a geometric
approach to the construction of K-homology classes of a
$C^*$-algebra $\A$: for the commutative $C^*$-algebra
of continuous functions on a compact smooth manifold,
the phase $F:=\frac{D}{|D|}$ of an
elliptic first order differential operator $D$ on a
vector bundle defines such a class,
and for a noncommutative algebra, the fundamental task is to represent $\A$ 
on a Hilbert space $\hH$ and to find a
self-adjoint operator $D$ that has compact resolvent
(so it is ``very'' unbounded) 
but at the same time has bounded commutators with the
elements of $\A$. 

In classical geometry, equivariant differential
operators on Lie groups 
provide examples that can be described purely in
terms of representation theory, so since the
discovery of quantum groups, many attempts were made to apply Connes'
programme to these $C^*$-algebras. The fact that their 
representation theory resembles that of their classical counterparts so closely 
allows one indeed to define straightforwardly an
analogue say of the Dirac
operator on a compact simple Lie group, but it turns
out to have unbounded commutators with the elements
of $\A$.

Many solutions to this conundrum were found and
studied, focusing on various approaches and motivations
ranging from index theory 
\cite{DLSSV1, DLSSV2, C-index, CP, KRS, NT, NT1} 
over the theory of covariant differential calculi 
\cite{H,S} to the Baum-Connes conjecture 
\cite{V-BCc}.
However, it seems
fair to say that there is still no sufficient general
understanding of how Connes' machinery applies to
algebras obtained by deformation quantisation in
general and quantum groups in particular.  

The aim of the present note is to use
the fundamental example of $\SU$ for
discussing yet another mechanism for
obtaining bounded commutators. In a
nutshell, the idea is to have a
representation of $\A \otimes \A^\op$
on $\hH$ and to use differential
operators $D$ with ``coefficients'' in
$\A^\op$ to achieve bounded commutators
with $\A$. Our starting point 
is a noncommutative
analogue of a dense coordinate chart
on $\SU$ that is compatible 
with the symplectic
foliation of the quantised Poisson
manifold $\SU$. The noncommutative
analogue is obtained by replacing a
complex unit disc by the quantum disc. 
We use a quantised differential calculus on 
this chart to 
define quantisations of left
invariant vector fields that act on the
function algebra by twisted derivations.
This is where it becomes necessary to
consider coefficients from $\A^\op$. 

We then build two twisted Dirac
operators using these twisted
derivations and show that they are
related by a gauge transformation 
that arises from a rescaling of the
volume form. A fruitful
direction of further research might be to
investigate the spectral and homological
properties of these and similar
operators.

\section{The Dirac operator on $\SU$} 
The $C^*$-algebra $\A$ we are going to consider is a
strict deformation quantisation 
of the algebra of continuous complex-valued 
functions on the Lie group 
$$
\SU=\left\{ \left( \begin{array}{rr}
\alpha\ & \beta \\
-\bar\beta\ & \bar\alpha \\
\end{array} \right) \mid   \alpha,\beta\in\C,\ \alpha\bar
\alpha + \beta \bar\beta = 1 \right\}
$$
that we identify as usual with $\S^3 \subset \C^2$,
identifying the above matrix with $(\alpha,\beta)$. 

We denote by 
$$
X_0 := \left( \begin{array}{rr}
\im\ & 0 \\
0\ & -\im \\
\end{array} \right),\quad 
X_1 :=  \left( \begin{array}{rr}
0\ & 1 \\
-1\ & 0 \\
\end{array} \right),  \quad 
X_2:= \left( \begin{array}{rr}
0\ & \im \\
\im\ & 0 \\
\end{array} \right) 
$$
the standard generators of the Lie algebra $\su$ of 
$\SU$, and, 
by a slight abuse of notation, also the corresponding 
left invariant vector fields on $\SU$. 

In this section, we describe the Dirac operator $D$ of
$\SU$ in local coordinates that are adapted to the
quantisation process. To define the local coordinates,   
consider the map 
\[\label{lc}
	\bar \disc \times \S^1 \rightarrow 
	\S^3,\quad
	(z,v) \mapsto 
	(z,\sqrt{1-z \bar z}\hs v),
\]
where $\disc:=\{ z\in \C : |z|<1\}$ is the open unit disc,
$\bar \disc$ is its closure, and $\S^1=\partial\disc$ 
is its boundary.  
Restricting the map to
$(\disc \times \S^1)\setminus(\disc \times\{-1\}) \cong 
\disc \times  (-\pi, \pi)$ 
defines a dense coordinate chart 
\[\label{x}
	\xx \colon \disc \times (-\pi, \pi) 
	\, \lra \, \S^3, \quad \xx(z,t)  := 
	(z, \sqrt{1-z\bar z}\, \E^{\im t}) 
\]
that is 
compatible with the standard 
differential structure on 
$\SU \cong \S^3$. 
The pull-back of the bi-invariant 
volume form on $\S^3$ assigns a measure to $\disc \times (-\pi,\pi)$ 
and the resulting Hilbert
space of $L_2$-functions will be denoted by $\hH$.

We write $f(z, t):= f\circ \xx(z,t)$ for functions 
$f$ on $\S^3$ and thus identify these with continuous
functions on $\bar\disc \times [-\pi,\pi]$ satisfying
the boundary conditions 
\[\label{CSU} 
 	f(u, t) = f(u, 0),\quad 
	f(z,-\pi)=
	f(z,\pi)\quad \forall \ u \in \S^1,\ z \in \disc, \ 
	t\in [-\pi, \pi]. 
\]
Let $\Gamma^{(1)}(\SU)$ denote the set of 
$C^{(1)}$-functions (continuously differentiable ones) on  
$\bar \disc \times [-\pi, \pi] $ 
satisfying \eqref{CSU}.  
The corresponding functions on $\S^3$ are not
necessarily $C^{(1)}$, but  
absolutely continuous, and can therefore be considered 
as belonging to the domain 
of the first order differential operators 
$X_0$, $X_1$ and $X_2$. 
Here, derivations
are understood to be taken in the weak
sense.
Therefore we may consider 
$$
H:= -\im X_0,  \quad E:= \half (X_1 - \im X_2), \quad F:= -  \half (X_1 +\im X_2)
$$
 as first order differential operators on $\hH$ with domain $\Gamma^{(1)}(\SU)$. 
 A direct calculation shows that these operators take in the parametrisation \eqref{x} the following form: 
 \begin{align}
H &= z \dz - \bar z \dbz + \im \dt,  \label{H} \\
E &= -  \sqrt{1-\bar z z} \,\E^{-\im t} \dbz - 
 \mbox{$\frac{\im}{2}$}\hs \mbox{$\frac{z}{\sqrt{1-\bar z z }}$}\hs  \E^{-\im t} \hs \dt,  \label{E} \\
F &=   \sqrt{1-\bar z z} \,\E^{\im t} \dz - 
 \mbox{$\frac{\im}{2}$}\hs \mbox{$\frac{\bar z}{\sqrt{1-\bar z z }}$}\hs  \E^{\im t} \hs \dt .  \label{F} 
\end{align}

Since $\SU$ is a Lie group, its tangent bundle is
trivial and hence admits a trivial spin structure.  
We consider 
$\Gamma^{(1)}(S) := \Gamma^{(1)}(\SU)\oplus \Gamma^{(1)}(\SU)$ 
as a vector space of differentiable sections (in the 
weak sense) of the associated spinor bundle. 
The Dirac operator 
with respect to the bi-invariant metric on $\SU$ 
is then given by the closure of 
$$
D:=\left( \begin{array}{cc}
H-2\ & E \\
F\ & -H-2 \\
\end{array} \right)\ : \ \Gamma^{(1)}(S) \subset \hH\oplus \hH \ \lra\  \hH\oplus \hH, 
$$
see e.g.~\cite{F}. 

\section{A representation of
quantum SU(2) by multiplication
operators} 


The quantised coordinate ring of 
$\SU$ at $q \in (0,1)$ is
the universal unital *-al\-ge\-bra $\Aq$ 
containing elements $a,c$
such that
\begin{eqnarray*}
&& 
        ac=qca,\quad
        ac^*=qc^*a,\quad
        cc^*=c^*c,\\
&&   
        aa^*+q^2cc^*=1,\quad a^*a+ cc^*=1.  
\end{eqnarray*}
It admits a faithful Hilbert space representation  $\rho$ on $ \lN\ot \lZ$ 
given on orthonormal bases $\{e_n\}_{n\in{\N}}\subset \lN$ and $\{b_k\}_{k\in\Z}\subset \lZ$ by 
\begin{align} 
& \rho(a) (e_n\ot b_k) = \sqrt{1 - q^{2n}} \hs e_{n-1} \ot b_k, \label{a} \\ 
& \rho(c) (e_n\ot b_k)  = q^n \hs e_{n} \ot b_{k-1}.  \label{c}
\end{align} 
The norm closure of the *-algebra generated by $ \rho(a),  \rho(c)\in B( \lN\ot \lZ)$  
is isomorphic to $\CSUq$,  the universal  C*-algebra of $\Aq$, see e.g.\ \cite{MNW}.  

The starting point of this paper is
a quantum counterpart to the chart
\eqref{lc}. To define it, let 
$z\in B(\lN)$ and $u\in B(\lZ)$ be given by 
\[ \label{zu}
 z\hs e_n := \sqrt{1 - q^{2n}} \hs e_{n-1}, \ \ n\in\N, \qquad u\hs b_k:= b_{k-1},\ \ k\in\Z, 
\] 
and set 
\[\label{y} 
y := \sqrt{1- z^*z}\in B(\lN). 
\]
Then $y$ is a positive self-adjoint trace class operator on $\lN$ acting by 
$$
y\hs e_n = q^n \hs e_n
$$ 
and satisfying the relations 
\[ \label{zy} 
 z y  = q y z, \quad y z^* = q z^* y.  
\]
Note that one can now rewrite Equations \eqref{a} and \eqref{c} as 
\[   \label{zyu}
\rho(a) = (z\ot 1)\hs (e_n \ot b_n), \qquad \rho(c) = (y\ot u) \hs (e_n \ot   b_n). 
\] 

The bilateral shift $u$ generates a commutative C*-subalgebra of $B(\lZ)$ which is isomorphic to $C(\S^1)$.
The operator $z\in B(\lN)$ satisfies the defining relation of the quantum disc algebra $\Uq$, 
\[
  zz^* - q^2 z^*z = 1-q^2.   \label{qdr}
\]
It is known \cite{KL} that the universal C*-algebra of the quantum disc $\Uq$, 
generated by a single generator and its adjoint satisfying 
\eqref{qdr}, is isomorphic to the Toeplitz algebra $\T$ which can also be viewed as the C*-subalgebra of $B(\lN)$ 
generated by the unilateral shift
\[ \label{so}
 s\hs e_n = e_{n+1} , \quad n\in\N. 
\] 
Moreover, the bounded operator $z$ defined in \eqref{zu} also generates the Toeplitz algebra $\T\subset B(\lN)$,  
and the so-called symbol map of the Toeplitz extension \cite{V} 
\begin{equation}   \label{ext}
\xymatrix{
 0\;\ar[r]&\; \K(\lN) \;\ar@{^{(}->}[r] &\; \T
\;\ar[r]^{\!\!\tau} &\; C(\S^1) \;\ar[r] &\,0}
\end{equation}
can be given by $\tau(z)=\tau(s)=u$, where $u$ denotes the unitary generator of $C(\S^1)$ 
and $\K(\lN)$ stands for the C*-algebra of compact operators on $\lN$.

Returning to the map \eqref{lc}, observe that $\SU\cong \S^3$ is homeomorphic to the 
topological quotient of $\bar \disc \times \S^1$ given by shrinking the circle $\S^1$ to a point 
on the boundary of $\bar \disc$. This can be visualised by the following push-out diagram: 
\begin{align*}\mbox{ }  \\[-18pt]
\xymatrix@=3mm@R=1mm{ 
&&\ \S^3&&\\
 \chi_1(z,v):=(z, \sqrt{1-z \bar z}\hs v)  \;&&&&\;
\chi_2(u):=(u,0)\\
&\bar\disc\times \S^1 \ar[ruu]^{\chi_1}
&&\S^1 \ar[luu]_{\chi_2}& \\
(\iota,\id)(u,v):=(u,v)\;&&&&\;\pr_1(u,v):=u.\\
&&\S^1 \times \S^1 \ar[ruu]_{\pr_1}\ar[luu]^{(\iota,\id)}&&
}
\end{align*}

Applying the functor that assigns to a topological space the algebra of continuous functions, 
we obtain a pull-back diagram of C*-algebras. 
Quantum SU(2) is now obtained by
replacing in this pull-back diagram   
the C*-algebra $C(\bar \disc)$ by the
Toeplitz algebra $C(\bar\disc_q):=\T$,
regarded as the algebra of continuous
functions on the quantum disc. 
The restriction map $\iota^*:  C(\bar \disc)\lra C(\S^1)$, \,$\iota^*(f)=f\!\!\upharpoonright_{\S^1}$, is 
replaced by the symbol map $\tau: \T\lra C(\S^1)$ from \eqref{ext}.
The resulting pull-back diagram has the following structure: 
\begin{equation*} 
\xymatrix{ \nonumber
& \makebox[48pt][c]{
$P\,:=\,\big(C(\qd)\otimes C(\S^1)\big) \hspace{0pt}{\underset{(\pi_1,\pi_2)}{\times}}\hspace{0pt} C(\S^1)
\hspace{74pt}\mbox{ }$}
\ar[dl]_{\mathrm{pr}_1} \ar[dr]^{\mathrm{pr}_2}& \\
C(\qd) \otimes C(\S^1) \ar[rd]_{\pi_1:=\tau \hs \ot\hs \id} &    &
C(\S^1) \ar[ld]^{\pi_2:=\id\hs \ot\hs 1}\\
 &  C(\S^1)\otimes C(\S^1)& 
}
\end{equation*}
Note that 
$ (t\ot f,g)\in P$
if and only if $\tau(t)\ot f \hsp=\hsp g \ot 1 \in C(\S^1)\ot \C\hs 1\subset C(\S^1)\ot C(\S^1)$. 
Since $\C\hs 1 \cong \C(\{\text{pt}\})$, the interpretation of $\tau(t)\ot f = g \ot 1$ is that, whenever we evaluate  
$t\in C(\qd)$ on the boundary, the circle $\S^1$ in $\qd\times \S^1$ collapses to a point.

Moreover, it can be shown \cite[Section 3.2]{HW} that $\mathrm{pr}_1$ yields an isomorphism of \mbox{C*-algebras} 
$P\cong \mathrm{pr}_1(P) \cong \CSUq
\subset B( \lN\ot \lZ)$
such that $\rho(a)=$ \mbox{$\mathrm{pr}_1((z\ot 1, u))$} and $\rho(c)= \mathrm{pr}_1( (y\ot u, 0))$, 
see \eqref{zyu}. 

Viewing $\CSUq$ as a subalgebra of $C(\qd) \otimes C(\S^1)$ allows us to 
construct the following 
faithful Hilbert space representation of 
the C*-algebra $\CSUq$ in which it acts by 
multiplication operators on a
noncommutative function 
algebra. This leads to an interpretation
as an algebra of 
integrable functions on the quantum space 
$\oqd \times \S^1 $.  First note that, since $z^*z= 1-y^2$ and $zz^*= 1- q^{2}y^2$, any 
element $p\in \Uq$ can be written as
$$
p=\sum_{n=0}^N z^{*n}  p_n(y) + \sum_{n=1}^M p_{-n} (y)\hs  z^{n} , \quad N,M\in\N, 
$$
with polynomials $p_n$ and $p_{-n}$. Using the functional calculus of the  self-adjoint 
operator $y$ with spectrum $\spec(y)= \{q^n: n\in\N\} \cup \{0\}$, we define 
$$
\Fq:= \left\{\sum_{n=0}^N z^{*n} f_n(y) + \sum_{n=1}^M f_{-n} (y)\hs  z^{n} : \ N,M\in\N, \ \ 
 f_k \in L_\infty(\spec(y))\right\}.
$$
Using the commutation relations 
$$
z f(y) = f(qy) z, \quad f(y) z^* = z^* f(qy), \quad f \in L_\infty(\spec(y)), 
$$
one easily verifies that $\Fq$ is a *-algebra. Let $s$ 
denote the unilateral shift operator on $\lN$ from Equation \eqref{so}. 
For all functions  $f \in L_\infty(\spec(y))$, 
it satisfies the commutation relations 
$$
s^* f(y) = f(qy) s^*, \quad f(y) s = s f(qy).
$$
Writing $z$ in its polar decomposition $z=s^*\hs |z| = s^* \sqrt {1-y^2}$, one sees that 
$$
\Fq= \left\{\sum_{n=0}^N s^n f_n(y) + \sum_{n=1}^M f_{-n} (y) s^{*n} : \ 
N,M\in\N, \ \  f_k \in L_\infty(\spec(y))\right\}. 
$$
Since $y^\alpha e_n = q^{\alpha n} e_n$, the operator $y^\alpha$ is trace class for all $\alpha >0$. 
Therefore the positive functional 
\[ \label{Tr} 
\int_{\oqd}(\,\cdot\,)\,\dd \mu_\a  \ :\  \Fq \,\lra\, \C, \qquad  \int_{\oqd}\! f \,\dd \mu_\a := (1-q) \Tr_{\lN}(fy^\alpha ) . 
\]
is well defined. Explicitly, 
it is given by 
$$
 \int_{\oqd} \left( \sum_{n=0}^N s^n f_n(y) + \sum_{n=1}^M f_{-n} (y) s^{*n} \right) \hsp \dd \mu_\a 
 = (1-q) \sum_{n\in\N} f_0(q) q^{\alpha n} .
$$
Using $ \Tr_{\lN}(s^n s^{*k} f(y)y) =0$ if $k\neq n$, one easily verifies that it is faithful. 
In terms of the Jackson integral $\int_0^1 f(y)\dd_q y= (1-q) \sum_{n\in\N} f(q^n)q^n$, we can write 
\begin{align} \nonumber
 &\int_{\oqd} \left( \sum_{n=0}^N s^n f_n(y) + \sum_{n=1}^M f_{-n} (y) s^{*n}\! \right) \hsp \dd \mu_\a   \\
 &\qquad \qquad \quad = \int_0^1\! \int_{-\pi}^{\pi}\sum_{n=0}^N \E^{\im n\phi} f_n(y) + \sum_{n=1}^M f_{-n} (y) \hs \E^{-\im n\phi}
 \hs \dd \phi\, y^{\alpha -1} \dd_q y. \label{measure} 
\end{align} 

Note that the commutation relation between $y^\alpha $ and functions from $\Fq$ can be expressed by 
the automorphism $\sigma^\alpha  : \Fq \lra \Fq$ given by 
\[ \label{s}
\s^\alpha (s) = q^{- \alpha } \hs s,\quad \s^\alpha(s^*) = q^\alpha  \hs s^*, \quad 
 \s^\alpha (f(y)) = f(y), \ \  f \in L_\infty(\spec(y)), 
\]
where $\alpha\in\R$. Then, for all $h,g\in \Fq$, 
$$
g \hs y^\alpha  = y^\alpha  \hs \s^\alpha (g),   
$$
and therefore 
\begin{align}  \nonumber
\int_{\oqd} \! gh \,\dd\mu_\a 
&= (1\hsp -\hsp q) \Tr_{\lN}(ghy^\alpha)  =  (1\hsp -\hsp q) \Tr_{\lN}(\s^\alpha (h)  g y^\alpha ) \\
&= \int_{\oqd} \!\s^\alpha(h)g  \,\dd\mu_\a.  \label{mod}
\end{align} 
Note that we also have 
\[  \label{astar} 
 (\s^\alpha( f))^* = \s^{-\alpha}(f^*), \quad f \in \Fq. 
\]

We use the faithful positive functional $\int_{\oqd} (\,\cdot\,) \, \dd \mu_\a$ to define an inner product on $\Fq$ by 
$$
\ip{f}{g} := \int_{\oqd} f^* g \,\dd\mu_\a . 
$$ 
The Hilbert space closure of $\Fq$ will be denoted by  $L_2(\oqd,\mu_\a)$. 
Left multiplication with functions $x\in \Fq$ defines a faithful *-representation of $\Fq$ on 
 $L_2(\oqd,\mu_\a)$ since 
 $$
 \ip{xf}{g} = \int_{\oqd} f^* x^*g \,\dd\mu_\a =  \ip{f}{x^*g}. 
 $$
Observe that $\Fq$ leaves the subspace 
\[   \label{F0}
\F_0(\oqd):= \left\{\sum_{n=0}^N s^n f_n(y) + \sum_{n=1}^M f_{-n} (y) s^{*n} \in \Fq : \ 
  \mathrm{supp}(f_k) \text{ is finite}   \right\} 
\] 
of $L_2(\oqd,\mu_\a)$  invariant. 
Since $\F_0(\oqd)$ contains an orthonormal basis (see \cite[Proposition 1]{W}), it is 
dense in $L_2(\oqd,\mu_\a)$. We extend $\Fq$ by the unbounded element $y^{-1}$ 
and define $\cO^+(\oqd)$ as the *-algebra generated by the operators $y^{-1}$ and all $f\in \Fq$,  
considered as operators on $\F_0(\oqd)$. 
Furthermore, let $\cO^+(\oqd)^{\op}$ denote the *-algebra obtained from $\cO^+(\oqd)$ by replacing the 
multiplication with the opposite one, i.e.\ $a\cdot b:= ba$. 
 Then we obtain 
 a representation of $\cO^+(\oqd)^{\op}$ on $\F_0(\oqd) \subset L_2(\oqd,\mu_\a)$ by right multiplication, 
 \[ \nonumber 
 a^{\op}f:= fa, \quad a\in \cO^+(\oqd)^{\op}, \ \ f \in \F_0(\oqd). 
 \]
 Clearly, this representation commutes with the operators of $\cO^+(\oqd)$, as these act by left multiplication. 
 However, it is not a *-representation. 
 More precisely, \eqref{mod} and \eqref{astar} give 
 $$
 \ip{x^{\op}f}{g}  =  \ip{fx}{g} = \int_{\oqd} x^* f^* g\,\dd\mu_\a =   \int_{\oqd}  f^* g \s^{-\alpha} (x^*)\,\dd\mu_\a =
 \ip{f}{(\s^\alpha(x)^*)^{\op}g}, 
 $$
therefore  
\[ \label{opstar} 
(x^{\op})^* =  (\s^\alpha(x)^*)^{\op} .
\] 
Note that $y>0$ and $\s(y) =y$ imply that the multiplication operators $y^{\beta}$ 
and $(y^{\beta})^{\op} = (y^{\op})^{\beta} $, $\beta \in \R$,  determine 
well defined (unbounded) self-adjoint operators on $L_2(\oqd,\mu_\a)$. 

Next we use the isomorphism $\lZ\cong L_2(\S^1)$ given by $b_n :=   \frac{1}{\sqrt{2\pi}} \E^{\im n t}$ and identify 
$u$ from \eqref{zu} with the multiplication operator $uf(t):= \E^{\im t} f(t)$. 
In this way we obtain a faithful *-representation 
$\trho : \Aq \lra B(L_2(\oqd,\mu_\a) \ot L_2(\S^1))$ by multiplication operators. On generators, it is given by 
$$
\trho(a)(f\ot g) := zf\ot g, \quad \trho(c) := yf\ot ug. 
$$
The closure of the image of $\trho$ is again isomorphic to $\CSUq$. 

\section{Quantised differential calculi} 

Taking as its domain the absolutely continuous functions $AC(\S^1)$ with the weak derivative in $L_2(\S^1)$, 
the partial derivative $ \im \dt$ becomes a self-adjoint operator on $L_2(\S^1)$ satisfying the Leibniz rule 
$$
 \im \dt (\varphi\hs g) = (\im \mbox{$\frac{\partial \varphi }{\partial t}$})g + \varphi (\im \dt g) , 
 \quad \varphi\in C^{(1)}(\S^1), \ \ g\in \dom( \im \dt) . 
$$

We consider a first order differential *-calculus $\dd : \Uq \lra \Omega(\oqd)$, where 
$\Omega(\oqd)= \dd z\hs   \Uq +  \dd z^* \Uq$ 
with $\Uq$-bimodule structure given by 
$$
\dd z\hs  z^* = q^2 z^*\hs  \dd z ,\  \  \dd z^* \hs  z = q^{-2} z\hs  \dd z^* ,\  \  
\dd z \hs z =q^{-2} z \dd z , \ \ \dd z^* \hs z^* =q^{2} z^* \dd z^*, 
$$
see \cite{KS} for definitions and background on differential calculi. 
With $\s^\a$ from \eqref{s}, it follows that 
$$
\dd z \hs f = \s^{-2}(f) \hs \dd z,\quad \dd z^* \hs f = \s^{-2}(f) \hs \dd z^*,\quad f\in \Uq.  
$$
We define partial derivatives $\dz$ and $\dbz$ by 
$$
\dd(f) =  \dd z \hs  \dz(f)   +\dd z^*\hs  \dbz(f)  , \quad f\in \Uq. 
$$
Recall that $y^2= 1- z^* z$ and  $zz^* -z^* z= (1-q^2) y^2$ by \eqref{qdr}  and \eqref{y}. 
Using 
$$
1 = \dz(z) =  \mbox{$\frac{-1}{1-q^2}$} \hs  y^{-2} [z^*,z], \quad 
  1 = \dbz(z^*) =  \mbox{$\frac{1}{1-q^2}$} \hs  y^{-2} [z,z^*] , 
$$
the Leibniz rule for the commutator 
and $y^{-2} p= \s^{2}(p) y^{-2}$ for all $p\in \Uq$, one verifies by direct calculations on 
monomials $z^n z^{*m}$  that 
\[ \nonumber 
\dz(p) =  \mbox{$\frac{-1}{1-q^2}$} \hs y^{-2} [z^*,p], \quad 
  \dbz(p) =  \mbox{$\frac{1}{1-q^2}$} \hs   y^{-2} [z,p] , \quad p\in \Uq. 
\] 
We extend the partial derivatives  $\dz$ and $\dbz$ to 
$$
\dFq:= \{ f\in\Fq: y^{-2}[z^*, f] \in \Fq, \ \  y^{-2} [z, f]   \in\Fq  \}
$$
by setting 
\[ \nonumber 
\dz (f) :=   \mbox{$\frac{-1}{1-q^2}$} \hs y^{-2} [z^*,f], \quad  
 \dbz(f) := \mbox{$\frac{1}{1-q^2}$} \hs  y^{-2}  \hs  [z,f], \quad f\in \dFq. 
\] 
Note that $\Uq\subset \dFq$. By the spectral theorem for functions in $y=y^*$, one readily 
proves that $\Uq$ is dense in $L_2(\oqd,\mu_\a)$.  Thus $\dz$ and $\dbz$ are 
densely defined linear operators on $L_2(\oqd,\mu_\a)$.  

Moreover, it is easily seen that the automorphism 
$\s^\alpha$ from \eqref{s} preserves $\dFq$. For instance,  
$y^{-2}[z^*,\s^\alpha(s^n f(y))] = q^{-\alpha n} y^{-2}[z^*,s^n f(y)]\in \Fq$ for all 
$s^n f(y) \in \dFq$. 
Similarly one shows that $\dFq$ is a *-algebra. For example, 
\begin{align*}
y^{-2} [ z^*, fg] & = y^{-2} [ z^*, f]g + y^{-2} f y^{-2}y^{-2}[ z^*,g]  \\
&= y^{-2} [ z^*, f]g + \s^2(f )\hs y^{-2}[ z^*,g]\in \Fq
\end{align*}
and 
$$
y^{-2} [ z^*, f^*]  =- (y^{-2}  y^{2} [ z, f] y^{-2})^* = -q^2 (y^{-2} [z, \s^2(f)])^* \in \Fq 
$$
for $f,g\in\dFq$.

\section{Twisted derivations} 

\subsection{Twist: $\pmb{\sigma^1}$} \mbox{ }\\[4pt]
\noindent 
Our aim is to replace the first order differential operators $H$, $E$ and $F$ from \eqref{H}-\eqref{F} 
by appropriate noncommutative versions.  
First we consider $q$-analogues of the operators $\sqrt{1-\bar z z} \dz $ 
and $\sqrt{1-\bar z z} \dbz $ and  define $T_i: \dFq \lra \Fq$, \,$i=1,2$, by 
\[ \label{defT1T2}
 T_1f:= y \dz f =  \mbox{$\frac{-1}{1-q^2}$} \hs y^{-1} [z^*,f], \quad 
 T_2f:= y \dbz f =  \mbox{$\frac{1}{1-q^2}$} \hs y^{-1} [z,f]. 
\]
Observe that $T_1$ and $T_2$ satisfy a twisted Leibniz rule: For all $f,g\in \dFq$, 
\begin{align}\nonumber
T_1(fg) &=  \mbox{$\frac{-1}{1-q^2}$} \hs y^{-1}[z^*,fg] 
=   \mbox{$\frac{-1}{1-q^2}$} \hs y^{-1}([z^*, f] g  + f y y^{-1} [z^*, g]) \\
&= (T_1 f)g + \s^{1}\hsp (f)\hs  T_1g \label{T1}
\end{align}
and similarly 
\[  \label{T2} 
T_2(fg) = (T_2 f)g + \s^{1}\hsp (f)\hs  T_2g. 
\] 
Setting $\hat T_1 := T_1 \ot 1$, $\hat T_2 := T_2 \ot 1$  and $\hat \s^1 := \s^1 \ot 1$, we get  

\[ \label{T1T2}
\hat T_1(\phi \psi ) = (\hat T_1 \phi)\psi  + \hat\s^{1}\hsp (\phi)\hs  \hat T_1\psi, \quad 
\hat T_2(\phi \psi ) = (\hat T_2 \phi)\psi  + \hat\s^{1}\hsp (\phi)\hs  \hat T_2\psi
\] 
for all 
$\phi, \psi \in \dFq \ot C^{(1)}(\S^1)$ by \eqref{T1} and \eqref{T2}. 

Next consider the operator 
$\hat T_0:= y^{-1} \dt$ on the domain 
$\dom(y^{-1}) \ot C^{(1)}(\S^1)$  in  $L_2(\oqd,\mu_\a) \ot L_2(\S^1)$. 
Note that, for all $f,g\in \Fq$ with $g\in\dom(y^{-1})$, one has 
$y^{-1}fg= \s^1(f)y^{-1} g\in L_2(\oqd,\mu_\a)$, hence $fg\in\dom(y^{-1})$. 
Now, for all $\varphi ,\, \xi \in C^{(1)}(\S^1)$, 
\begin{align}  \nonumber 
\hat T_0(fg\ot \varphi\xi) 
&= y^{-1} fg\ot \dt ( \varphi \xi ) \nonumber \\
&= y^{-1} fg\ot (\dt \varphi)\xi +  y^{-1} fy\hs  y^{-1}  g\ot  \varphi\,(\dt \xi) \nonumber \\
&= \big(\hat T_0(f\ot \varphi)\big) (g\ot\xi) + (\s^1(f) \ot \varphi) \big(\hat T_0 (g\ot\xi)\big). 
 \nonumber 
\end{align}
Therefore, for all $\phi, \psi \in \Fq \ot C^{(1)}(\S^1)$  with $\psi \in \dom (y^{-1} \ot 1)$, we have 
\[ \label{T0}
\hat T_0 (\phi \psi) = (\hat T_0\phi)  \psi + \hat\s^1(\phi) \hs ( \hat T_0\psi). 
\]
As a consequence, $\hat T_0,$ $\hat T_1$
and $\hat T_2$ satisfy the same twisted Leibniz rule. 

In the definition of the Dirac operator, we will multiply $\hat T_0$, $\hat T_1$ and $\hat T_2$  
with multiplication operators from the
opposite algebra. The following lemma shows that 
these multiplication operators do not change the twisted Leibniz rule.  
Our aim is to prove that the Dirac operator has bounded twisted commutators 
with functions of an appropriate *-algebra, where the twisted commutator 
of densely defined operators $T$ on  $L_2(\oqd,\mu_\a)\bar\ot L_2(\S^1)$ with  
$\phi\in \Fq \ot C(\S^1)$ is defined by 
$$
[T, \phi]_{\s^1} := T \phi - \s^1(\phi) T. 
$$

The purpose of the following lemma is to 
clarify the setup for the algebraic
manipulations to be carried out, and to
ensure that these make sense in their
Hilbert space realisation. 

\begin{lem} \label{lem1}  
Let $\A$ be a unital *-algebra, $\int : \A \ra \C$ 
a faithful positive functional, $\hH$ the Hilbert space closure of $\A$ 
with respect to the inner product
$\ip{a}{b}:=\int a^*b$, and assume that  
left multiplication by an element in $\A$
defines a bounded operator on $\hH$. 
Let $T$ be a densely defined linear operator on $\hH$, 
$\A^1\subset \A$ a *-sub\-algebra and $\cD\subset \hH$
a dense subspace such that $\cD + \A^1\subset\dom(T)$,  
$T(\A^1)\subset\A$, and $T$ satisfies the twisted 
Leibniz rule 
$$
T(f \psi) = (T f)\psi + \s(f)\hs  T\psi,\quad f\in \A^1, \ \ \psi\in \cD + \A^1
$$
for an automorphism $\s:\A\ra\A$. Assume
finally that $\hat x$ is a densely defined 
linear operator on $\hH$ with 
$\cD\subset \dom(\hat x)$, and that 
$f\psi \in \cD$  and  $\hat x f \psi =
f\hat x \psi$ hold for all $f\in\A$ and 
$\psi\in \cD$. 
Then 
\[ \label{Tfg}
\hat x T(fg) = \hat x (Tf) g + \s(f) \hat x T(g), \quad f,g\in \A^1
\]
as operators on $\cD$ and 
$$
[\hat x T, f]_{\s} \psi = \hat x (Tf)\psi, \quad f\in \A^1, \ \ \psi \in \cD.
$$
\end{lem} 

\begin{proof}
The only slightly nontrivial statement
is that each term in the following
algebraic computations is well-defined
as an operator on the domain $\cD$:
Let $\psi \in\cD$ and $f,g\in \A^1$. 
From the twisted Leibniz rule, we get 
\begin{align*}
( T(fg) )\psi &= T(fg\psi) - \s(fg) T\psi 
= (T f)(g\psi) + \s(f) T(g\psi) - \s(f) \s(g) T\psi \\
&= (T f)(g\psi) + \s(f) (T g)\psi  . 
\end{align*}  
Since  $\s(f )\in  \A$ 
for all $f \in \A$, it follows  that 
$$
\hat x( T(fg) )\psi =  \big(\hat x (T f)g  + \s(f)\hs \hat x (T g)\big)\psi , 
$$
which proves \eqref{Tfg}.  As 
$
\hat x\hs  T (f \psi) = \hat x\hs (Tf) \psi  +  \hat x \hs  \s(f)\hs T \psi 
= \hat x\hs (Tf) \psi  +    \s(f)\hat x\hs T \psi , 
$
we also have  $[\hat x\hs T, f]_{\s} \psi 
= \hat x\hs (Tf) \psi +    \s(f)\hs \hat x \hs T \psi -   \s(f)\hs  \hat x\hs T\psi
= \hat x\hs ( T f) \psi $. 
\end{proof}

By \eqref{T1T2} and \eqref{T0}, the lemma applies in particular to the operators $\hat T_0$, $\hat T_1$ 
and $\hat T_2$ with $\A:= \Fq \ot C(\S^1)$,
$\A^1:=\dFq \ot C^{(1)}(\S^1)$, 
$\hH:=L_2(\oqd,\mu _\alpha) \bar\otimes 
L_2(\S^1)$ (where integration on 
$\S^1$ is taken with respect to the
Lebesgue measure), the  
automorphism $\hat \s^1$
and the operators $\hat x $ coming from 
$\cO^+(\oqd)^\op$.  
As the dense domain, we may take 
\[\label{defbereich3}
\cD:= \F_0(\oqd) \ot C^{(1)}(\S^1).
\]

\subsection{Twist: $\pmb{\sigma^2}$} 
\label{sec2}\mbox{ }\\[4pt]
\noindent 
First we show that $\dz$ and $\dbz$ satisfy a twisted Leibniz rule for the automorphism $\s^2$. Let $f,g\in \dFq$. Then 
\begin{align}\label{S1}
\dz(fg) &=  
	\mbox{$\frac{-1}{1-q^2}$} 
	\hs y^{-2}[z^*\hsp,fg] 
	=   
	\mbox{$\frac{-1}{1-q^2}$} 
	\hs y^{-2}([z^*\hsp, f] g  
	+ f y^2 y^{-2} [z^*\hsp, g]) 
	\nonumber\\
&= (\dz f)g + \s^{2}\hsp (f)\hs  \dz g
\end{align} 
and similarly 
\[  \label{S2} 
\dbz(fg) = (\dbz f)g + \s^{2}\hsp (f)\hs  \dbz g. 
\] 

Setting $\hat S_1 := \dz \ot 1$, $\hat S_2 := \dbz \ot 1$  and $\hat \s^2 := \s^2 \ot 1$, we get for all 
$\phi, \psi \in \dFq \ot C^{(1)}(\S^1)$ 
\[ \nonumber 
\hat S_1(\phi \psi ) = (\hat S_1 \phi)\psi  + \hat\s^{2}\hsp (\phi)\hs  \hat S_1\psi, \quad 
\hat S_2(\phi \psi ) = (\hat S_2 \phi)\psi  + \hat\s^{2}\hsp (\phi)\hs  \hat S_2\psi
\] 
by  \eqref{S1} and \eqref{S2}. 

Next consider the operator 
$\hat S_0:= y^{-2} \dt$ on the domain 
$\dom(y^{-2}) \ot C^{(1)}(\S^1)$  in  $L_2(\oqd,\mu_\a) \ot L_2(\S^1)$. 
Again $fg\in\dom(y^{-2})$ for all $f\in \Fq$ and $g\in\dom(y^{-2})$ since 
$y^{-2}fg= \s^2(f)y^{-2} g\in L_2(\oqd,\mu_\a)$. 
Now, for all $\varphi ,\, \xi \in C^{(1)}(\S^1)$, 
\begin{align}  \nonumber 
\hat S_0(fg\ot \varphi\xi) 
&= y^{-2} fg\ot \dt ( \varphi \xi ) \nonumber \\ \nonumber 
&= y^{-2} fg\ot (\dt \varphi)\xi +  y^{-2} fy^2\hs  y^{-2}  g\ot  \varphi\,(\dt \xi) \\
&= \big(\hat S_0(f\ot \varphi)\big) (g\ot\xi) + (\s^2(f) \ot \varphi) \big(\hat S_0 (g\ot\xi)\big). 
\nonumber 
\end{align}
Therefore, for all $\phi, \psi \in \Fq \ot C^{(1)}(\S^1)$  with $\psi \in \dom (y^{-1} \ot 1)$, we have 
\[ \nonumber 
\hat S_0 (\phi \psi) = (\hat S_0\phi)  \psi + \hat\s^2(\phi) \hs ( \hat S_0\psi). 
\]
As a consequence, $\hat S_0,$ $\hat S_1$ and $\hat S_2$ satisfy the same twisted Leibniz rule 
for the twist $\hat\s^2$, and so do $x_i^\op S_i$ for $x_i^\op\in\{y^{\op}, (y^{\op})^2, z^\op, z^{*\op} \}$ 
by Lemma \ref{lem1}.

\section{Adjoints} 

\subsection{$\pmb{\alpha = 2}$}\mbox{ }\\[4pt]
\noindent 
Set $\alpha = 2$ in \eqref{Tr}, 
$$
\cD_0:=\{ f\in  \dFq\cap \dom( y^{-1})\cap \dom((y^{-1})^\op) : T_1(f), T_2(f) \in \dom((y^{-1})^\op)\} 
$$
and 
\[\label{defbereich} 
\cD:= \cD_0 \ot C^{(1)}(\S^1)  \;\subset \; L_2(\oqd,\mu_\a) \;\bar\ot \;L_2(\S^1). 
\]
It follows from $\F_0(\oqd)\subset \cD_0$ 
that $\cD$ is dense in $L_2(\oqd,\mu_\a) \;\bar\ot \;L_2(\S^1)$.  
Let $T_1$ and $T_2$ be the operators from \eqref{defT1T2} with domain $\cD_0$. 
Using \eqref{zy} and the trace property, 
we compute for all $f,g\in \cD_0$, 
\begin{align*} 
& \ip{ T_1 f}{g}  = -\mbox{$\frac{1-q}{1-q^2}$}\,  \Tr_{\lN}( f^*z y^{-1} g y^2 - z f^* y^{-1} g y^2   )\\
& =- \mbox{$\frac{1-q}{1-q^2}$}\,\Tr_{\lN}( q^{-1}  f^*y^{-1} z g y^2 -  q^{-2}  f^* y^{-1} g z y^2) \\
& =\mbox{$\frac{(1-q)(q^{-2}-q^{-1}) }{1-q^2}$}\,\Tr_{\lN}( f^*y^{-1}  g z y^2)
 - \mbox{$\frac{(1-q)q^{-1}}{1-q^2}$}\,
\Tr_{\lN}(   f^*y^{-1} z g y^2 -  f^* y^{-1} g z y^2) \\
& =  \ip{ f}{ (\mbox{$\frac{q^{-2}}{1+q}$} z^{\op} y^{-1}   - q^{-1} T_2 )g}, 
\end{align*} 
therefore 
\[ \label{T1*} 
\mbox{$\frac{q^{-2}}{1+ q}$} z^{\op} y^{-1}   - q^{-1} T_2  \ \subset \ T_1^*. 
\]
From \eqref{opstar}, it also follows that 
$$
\mbox{$\frac{q}{1+ q}$} z^{*\op} y^{-1}   - qT_1 \subset   T_2^*. 
$$
Similarly, using $zz^* -z^* z= (1-q^2) y^2$, 
\begin{align*} 
& \ip{ (zy^{-1})^{\op}T_1 f}{g}  = -\mbox{$\frac{1-q}{1-q^2}$}\,  
\Tr_{\lN}(y^{-1}z^* f^*z y^{-1} g y^2 - y^{-1} z^* z f^* y^{-1} g y^2   )\\
&  = -\mbox{$\frac{1-q}{1-q^2}$}\,  
\Tr_{\lN}(y^{-1} (z z^* -z^*z) f^* y^{-1} g y^2  + y^{-1}z^* f^*z y^{-1} g y^2 - y^{-1} z z^*  f^* y^{-1} g y^2   )\\
& = -(1-q) \,  \Tr_{\lN} \big( f^* y^{-1} gy y^2 \big) 
-\mbox{$\frac{1-q}{1-q^2}$}\,  \Tr_{\lN} \big(f^*y^{-1}(zg-gz) z^*y^{-1} y^2\big) \\
&=  \ip{f}{\big(-\s^1  -(z^*y^{-1})^\op T_2\big)g}, 
\end{align*} 
where $\s^1(g) = y^{-1}g  y= y^{-1} y^{\op} g$ for all $g\in\Fq\cap\dom(y^{-1})$. Hence 
\[  \label{T21}
-\s^1  -(z^*y^{-1})^\op T_2 \ \subset \big((zy^{-1})^{\op}\hs T_1\big)^* . 
\]
Since $y^{-1}$ and $ y^{\op} $ are
self-adjoint and thus $\s^1$ is
symmetric, we also get from the above calculations
\[  \label{T12}
-\s^1  -(zy^{-1})^\op T_1 \ \subset \big((z^*y^{-1})^{\op}\hs T_2\big)^* . 
\]  

Recall that $\im \dt$ is a symmetric operator on $C^{(1)}(\S^1)\subset L_2(\S^1)$. Also, for all 
$\varphi \in C^{(1)}(\S^1)$, we have 
\[  \label{edt}
 (\E^{\im t} \im \dt)^* \varphi =  \im \dt (\E^{-\im t} \hs \varphi ) 
 =  \E^{-\im t} \hs \varphi  +  \E^{-\im t}\hs  \im \dt \varphi. 
\]
From this, \eqref{opstar} and the self-adjointness of $y^{-1}$, it follows that 
$$
   q^{-2} z^{\op}  y^{-1}   \E^{-\im t} +  q^{-2} z^{\op}  y^{-1}   \E^{-\im t}  \im \dt \ \subset \  
    (  z^{*\op} y^{-1}  \E^{\im t} \im \dt )^*, 
$$ 
where the left-hand side and  $z^{*\op} y^{-1}  \E^{\im t} \im \dt $ are operators  on $\cD$. Analogously, 
\[ \label{zdt*}
  - q^{2} z^{*\op}  y^{-1}   \E^{\im t} +   q^{2} z^{*\op}  y^{-1}   \E^{\im t} \im  \dt \ \subset \  
   (  z^{\op} y^{-1}  \E^{-\im t} \im  \dt )^* .
\] 
Now we are in a position to state the main result of this section. 

\begin{prop}  \label{P1}
Consider the following operators on
$L_2(\oqd,\mu_\a) \;\bar\ot \;L_2(\S^1)$
with domain $\cD$ defined in
(\ref{defbereich}) above:  
\begin{align*}  
\hat H &:= (z y^{-1})^{\op}\hs y\hs  \dz  -  (z y^{-1})^{*\op}\hs y\hs  \dbz  +  \hs y^{\op}\hs   y^{-1} \hs \im \dt ,   \\ 
\hat E &:= -  \E^{-\im t}\hs   y \hs \dbz  -   
\mbox{$\frac{q^{-1}}{1+q}$}  z^{\op} \hs y^{-1} \hs  \E^{-\im t}\hs \im  \dt , \\
\hat F &:=  q\hs  \E^{\im t}\hs   y \hs \dz -   \mbox{$\frac{q}{1+q}$}  z^{*\op} \hs y^{-1} \hs  \E^{\im t}\hs \im  \dt . 
\end{align*}
Then \,$\hat H \subset \hat H^*$, \,$\hat F\subset \hat E^*$ and \,$\hat E\subset \hat F^*$. 
\end{prop} 
\begin{proof}
Since $y^{\op}$, \,$y^{-1}$ and $\im \dt$ are commuting symmetric operators on $\cD$, we have 
$y^{\op}\hs   y^{-1} \hs \im \dt \subset (y^{\op}\hs   y^{-1} \hs \im \dt)^*$. 
Now it follows from \eqref{T21} and \eqref{T12} that 
$$
\hat H^*\ \supset\  -\s^1  -(z^*y^{-1})^\op \hs y\hs  \dbz + \s^1 
 + (zy^{-1})^\op \hs y\hs  \dz + \hs y^{\op}\hs   y^{-1} \hs \im \dt \ = \ \hat H. 
$$ 
Furthermore, from \eqref{T1*}  and \eqref{zdt*}, we obtain 
$$
\hat F^* \ \supset\  \E^{-\im t} \mbox{$\frac{q^{-1}}{1+ q}$} z^{\op} y^{-1}   -  \E^{-\im t} \hs   y \hs \dbz
-   \E^{-\im t} \mbox{$\frac{q^{-1}}{1+q}$} \hs  z^{\op}  y^{-1}  
 -  \mbox{$\frac{q^{-1}}{1+q}$} \hs  z^{\op}  y^{-1}   \E^{-\im t}  \im \dt 
\  = \ \hat E . 
$$
The last relation also shows that $\hat
F^*$ is densely defined, thus 
$\hat F \subset \hat F^{**} \subset \hat E^*$. 
\end{proof}

\subsection{$\pmb{\alpha = 1}$} \mbox{ }\\[4pt]
\noindent 
Consider now 
\[\label{defbereich2} 
	\cD:= 
	\cD_0 \ot C^{(1)}(\S^1)  
	\;\subset \; L_2(\oqd,\mu_\a) 
	\;\bar\ot \;L_2(\S^1) 
\]
with $
	\cD_0:=
	\dFq\cap 
	\dom( y^{-2})$. 
For all $f,g \in \dFq$, 
\begin{align*} 
 \ip{ y^\op \dz f}{g}  &= -\mbox{$\frac{1-q}{1-q^2}$}\,  \Tr_{\lN}( y f^*z y^{-2} g y - y z f^* y^{-2} g y   )\\
& =-  q^{-2}\,\mbox{$\frac{1-q}{1-q^2}$}\,\Tr_{\lN}( f^*y^{-2} z g y^2 -  f^* y^{-2} g z y^2) \\
& =  - q^{-2}\,\mbox{$\frac{1-q}{1-q^2}$}\,\Tr_{\lN}( f^* y^{-2}[ z, g] y^2 )  \\
& =  \ip{ f}{ -q^{-2} y^\op \dbz  g},
\end{align*} 
thus  
\[  \label{yopdz}
-q^{-1} y^\op \dbz  \subset (q y^\op \dz)^*\quad \text{and}\quad 
 q y^\op \dz \subset (q y^\op \dz)^{**} \subset (-q^{-1} y^\op \dbz)^*.
\]
Next, 
\begin{align} \nonumber
 &\ip{ z^\op \dz  f}{g}  =  -\mbox{$\frac{1-q}{1-q^2}$}\,  \Tr_{\lN}( z^* f^*z y^{-2} g y - z^* z f^* y^{-2} g y   )\\ 
 &\quad =  -\mbox{$\frac{1-q}{1-q^2}$}\,  \Tr_{\lN}( q^{-1} f^*y^{-2} z g z^* y -  f^* y^{-2} g z^* z y   )\nonumber\\
&\quad =  - \mbox{$\frac{1-q}{1-q^2}$}\,  \Tr_{\lN}\big( q^{-1}( f^*  y^{-2} z g  z^* y -  f^* y^{-2}  g  z z^* y) \nonumber\\
&\quad \mbox{ } \hspace{8pt}   + q^{-1}  f^* y^{-2}  g  z z^* y  -  f^* y^{-2}  g  z^* z y  \big ) \nonumber\\
 &\quad = - \ip{  f}{(q^{-1}z^{*\op}  \dbz )g}  -q^{-1} \mbox{$\frac{(1-q)^2}{1-q^2}$}\, 
  \Tr_{\lN}\big( f^* y^{-2}  g  y +q f^* y^{-2}  g  y^2 y\big)\label{zopdz} 
\end{align} 
and 
\begin{align}  \nonumber
 \ip{ (z^{*\op}  \dbz ) f}{g}  &=  \mbox{$\frac{1-q}{1-q^2}$}\,  \Tr_{\lN}( z f^*z^* y^{-2} g y - z z^* f^* y^{-2} g y   )\\ 
  &=  \mbox{$\frac{1-q}{1-q^2}$}\,  \Tr_{\lN}( q f^*  y^{-2} z^* g  z y -  f^* y^{-2}  g  z z^* y   )\nonumber\\ 
   &=   \mbox{$\frac{1-q}{1-q^2}$}\,  \Tr_{\lN}\big( q( f^*  y^{-2} z^* g  z y -  f^* y^{-2}  g  z^* z y)
   \nonumber\\
   &\hspace{12pt}   + q  f^* y^{-2}  g  z^* z y -  f^* y^{-2}  g  z z^* y  \big )  \nonumber \\
   &= - \ip{  f}{(qz^{\op}  \dz )g} - \mbox{$\frac{(1-q)^2}{1-q^2}$}\, 
    \Tr_{\lN}\big( f^* y^{-2}  g  y +q f^* y^{-2}  g  y^2 y\big). \label{starzopdz}
\end{align} 
From \eqref{zopdz} and \eqref{starzopdz}, 
\[ \label{qzdz}
\ip{ (q z^\op \dz - z^{*\op}  \dbz )   f}{g}  = \ip{ f}{(q z^\op \dz - z^{*\op}  \dbz ) g}  
\] 
i.e., the operator $q z^\op \dz - z^{*\op}  \dbz $ is symmetric. As $(y^{2})^{\op}\hs  y^{-2} \im \dt$ is the product of 
commuting symmetric operators, 
\[ \label{yydt}
(y^{2})^{\op}\hs  y^{-2} \im \dt \ \subset\  \big((y^{2})^{\op}\hs  y^{-2} \im \dt\big)^*
\] 
is also symmetric. 

By \eqref{opstar} and \eqref{edt}, 
$$
 z^{\op} y^\op y^{-2} \E^{-\im t}   + z^{\op} y^\op y^{-2} \E^{-\im t}  \im \dt\  \subset\  
 (  z^{*\op} y^\op y^{-2} \E^{\im t} \im \dt)^*.
$$
Similarly, 
$$
 - z^{*\op} y^\op y^{-2} \E^{\im t} 
   + z^{*\op} y^\op y^{-2} \E^{\im t}  \im \dt \ \subset \ ( z^\op y^\op y^{-2} \E^{-\im t} \im \dt)^* . 
$$
Since $z^{\op} y^\op  \subset (z^{*\op} y^\op)^*$ by \eqref{opstar}, 
\[ \label{ef}
 \half z^{\op} y^\op y^{-2} \E^{-\im t}   + z^{\op} y^\op y^{-2} \E^{-\im t}  \im \dt\  \subset\  
 (- \half z^{*\op} y^\op y^{-2} \E^{\im t}    +  z^{*\op} y^\op y^{-2} \E^{\im t} \im \dt)^*.
\]
Analogously, using $z^{*\op} y^\op  \subset (z^{\op} y^\op)^*$, 
\[ \label{ee}
 -\half  z^{*\op} y^\op y^{-2} \E^{\im t}   + z^{*\op} y^\op y^{-2} \E^{\im t}  \im \dt \ \subset \ 
 ( \half z^{\op} y^\op y^{-2} \E^{-\im t}   + z^\op y^\op y^{-2} \E^{-\im t} \im \dt)^* . 
\]
As in the previous section, we summarise our results in a proposition. 
\begin{prop}  \label{P2} 
Let $\gamma_q$ be a non-zero real number. 
Consider the following operators on
$L_2(\disc_q) \;\bar\ot \;L_2(\S^1)$
with domain $\cD$ defined in
(\ref{defbereich2}) above:  
\begin{align*}  
\hat H_1 &:= q z^{\op}   \dz  -  z^{*\op} \dbz  + (y^{2})^{\op}\hs y^{-2} \im \dt,   \\ 
\hat E_1 &:= - q^{-1}  \E^{-\im t}\hs   y^\op \hs \dbz  - \gamma_q  z^{\op} y^\op y^{-2} \E^{-\im t}  \im \dt
 - \mbox{$\frac{\gamma_q}{2}$} z^{\op} y^\op y^{-2} \E^{-\im t}   ,   \\
\hat F_1 &:=  q\hs  \E^{\im t}\hs   y^\op \hs \dz  - \gamma_q  z^{*\op} y^\op y^{-2} \E^{\im t}  \im \dt
+ \mbox{$\frac{\gamma_q}{2}$} z^{*\op} y^\op y^{-2} \E^{\im t}   .
\end{align*}
Then \,$\hat H_1 \subset \hat H_1^*$, \,$\hat F_1\subset \hat E_1^*$ and \,$\hat E_1\subset \hat F_1^*$. 
\end{prop} 
\begin{proof}
$\hat H_1 \subset \hat H_1^*$ follows from \eqref{qzdz} and \eqref{yydt},  
\,$\hat E_1\subset \hat F_1^*$ follows 
from \eqref{yopdz} and \eqref{ef}, 
and \,$\hat F_1\subset \hat E_1^*$ follows
from \eqref{yopdz} and \eqref{ee}. 
\end{proof}

\section{The Dirac operator} 

Classically, the left invariant vector fields $H$, $E$ and $F$ act as first order differential operators 
on differentiable functions on $\SU$. 
In the noncommutative case, we will use the actions of $\hat H$, $\hat E$ and $\hat F$ to define an algebra of
 differentiable functions. 
 
 For $\hat x_i = x_i^{\op} \ot \eta_i \in \cO^+(\oqd)^{\op}  \ot C^\infty (\S^1)$, $i=0,1,2$,  consider the action on 
 $f\ot \varphi  \in \dFq\ot C^{(1)}(\S^1)$ 
 given by 
\begin{align}  \nonumber 
(\hat x_0\hat T_0 + \hat x_1\hat T_1 &+ \hat x_2\hat T_2 )( f\ot \varphi) \\     \label{action}
 &:= y^{-1} f\hs x_0 \ot \eta_0 \hs \dt(\varphi) +   T_1(f) \hs x_1\ot \eta_1 \hs \varphi +   T_2(f)\hs x_2 \ot \eta_2 \hs \varphi, 
\end{align} 
where the right-hand side of
\eqref{action} is understood as an
unbounded operator on
$L_2(\oqd,\mu_\a) \,\bar\ot\,
L_2(\S^1)$ with domain of definition
containing the subspace $\cD$ introduced 
in (\ref{defbereich3}). 
Note that the operators $\hat H$, $\hat E$ and $\hat F$ are of the form 
described in \eqref{action}, and that the operators  $\hat x_i$ and $\hat T_i$ satisfy the assumptions of Lemma \ref{lem1}. 
We define 
\begin{align}  \nonumber 
&\Gamma^{(1)}(\mathrm{SU}_q(2)) 
 := \{ \phi \in \dFq \ot C^{(1)}(\S^1) : \\   \label{Gamma}
&\hspace{70pt} \hat H(\phi),\ \hat E(\phi), \ \hat F(\phi) ,\  \hat H(\phi^*),\ \hat E(\phi^*), \ \hat F(\phi^*) \ \text{are bounded}\}. 
\end{align} 
From \eqref{T1}, \eqref{T2}, \eqref{T0} and Lemma \ref{lem1}, it follows that 
\[ \label{hatT}
\hat T (\varphi \psi) = (\hat T\varphi)  \psi + \hat\s^1(\varphi) \hs ( \hat T\psi), 
\]
for all $\varphi, \psi \in \Gamma^{(1)}(\mathrm{SU}_q(2))$  and $\hat T\in \{\hat H, \hat E,\hat F\}$. In particular, 
$\hat T (\varphi \psi)$ is again bounded so that $\Gamma^{(1)}(\mathrm{SU}_q(2))$ is a *-algebra. 

Finally note that the classical limit of $\hat H$, $\hat E$ and $\hat F$ for $q\ra 1$  is  formally  $H$, $E$ and $F$, 
respectively. This will also be the case if we rescale $\hat E$ and $\hat F$ by a real number $c=c(q)$ such 
that $\lim_{q\ra 1} c(q)=1$. 
Such a rescaling might be useful in later computations of the spectrum of the Dirac operator. 

\begin{thm}   \label{Thm1} 
Let $\alpha =2$. 
Set $\hH := (L_2(\oqd,\mu_\a) \bar\ot L_2(\S^1))\oplus (L_2(\oqd,\mu_\a) \bar\ot L_2(\S^1))$ and define 
$$
\pi :\Gamma^{(1)}(\mathrm{SU}_q(2))  \lra B(\hH), \quad \pi(\phi):= \phi \oplus \phi 
$$
as left multiplication operators. Then, for any $c\in\R$,  the operator  
$$
D:=\left( \begin{array}{cc}
\hat H-2\ & c \hat E \\
c \hat F\ & -\hat H-2 \\
\end{array} \right)
$$
is symmetric on 
$\cD \oplus \cD$ with 
$\cD$ from (\ref{defbereich}).
Furthermore,  
$$
[D, \pi(\phi)]_{\hat\s^1} : =D\hs  \pi(\phi) - \pi(\hat\s^1(\phi)) \hs D
$$
is bounded for all $\phi\in \Gamma^{(1)}(\mathrm{SU}_q(2))$. 
\end{thm}

\begin{proof} 
$D\subset D^*$ follows from Proposition \ref{P1}. 
For all $\phi\in \Gamma^{(1)}(\mathrm{SU}_q(2))$ and $\psi_1\oplus\psi_2$ in the domain of $D$, 
\begin{align*}
[D, \pi(\phi)]_{\hat\s^1} (\psi_1\oplus\psi_2) 
&= \hat H (\phi \psi_1) - \hat\s^1(\phi) \hat H ( \psi_1) + c(\hat E (\phi \psi_2) - \hat\s^1(\phi) \hat E ( \psi_2))\\
 &\hspace{10pt} \oplus \,  c(\hat F (\phi \psi_1) - \hat\s^1(\phi) \hat F ( \psi_1))  
   -(\hat H (\phi \psi_2) - \hat\s^1(\phi) \hat H ( \psi_2))  \\
& =  (\hat H\phi) \psi_1 + c (\hat E \phi) \psi_2 \ \oplus \ c(\hat F \phi)  \psi_1 - (\hat H\phi )\psi_2
\end{align*} 
by \eqref{hatT}, so $[D, \pi(\phi)]_{\s^1}$ is bounded by the definition of $\Gamma^{(1)}(\mathrm{SU}_q(2))$. 
\end{proof}

\begin{thm} \label{Thm2} 
Let $\hH$ and $\pi$ be defined as in Theorem \ref{Thm1} but with the measure on $\disc_q$ 
given by setting $\alpha=1$ in \eqref{Tr}. 
Let $\hat H_1$, $\hat F_1$ and $\hat
E_1$ be defined as in Proposition \ref{P2} 
and $\Gamma^{(1)}(\mathrm{SU}_q(2))$ as in \eqref{Gamma} with $\hat H$, $\hat F$ and  $\hat E$  replaced 
by $\hat H_1$, $\hat F_1$ and $\hat E_1$, respectively.  
Then, for any $c\in\R$,  the operator  
$$
D_1:=\left( \begin{array}{cc}
\hat H_1-2\ & c \hat E_1 \\
c \hat F_1\ & -\hat H_1-2 \\
\end{array} \right)
$$
is symmetric on $\cD \oplus \cD$ with 
$\cD$ from (\ref{defbereich2}).
Furthermore, 
$$
[D_1, \pi(\phi)]_{\hat\s^2} : =D_1\hs  \pi(\phi) - \pi(\hat\s^2(\phi)) \hs D_1
$$
is bounded for all $\phi\in \Gamma^{(1)}(\mathrm{SU}_q(2))$. 
\end{thm}
\begin{proof} 
Using the results for $\hat S_0$, \,$\hat S_1$ and \,$\hat S_2$ from Section \ref{sec2} 
and Proposition \ref{P2}, 
the proof is essentially the same as the proof of the previous theorem. 
\end{proof}

To view the Dirac operator of Theorem
\ref{Thm2} as a deformation of the classical Dirac operator, 
one may choose a continuously varying positive real number $\gamma_q$ such that $\lim_{q\ra 1}\gamma_q=\half$. 
For instance, if $\gamma_q := \frac{q}{1+q}$, then the Dirac operator of Theorem \ref{Thm2} resembles the one 
of Theorem \ref{Thm1}, the main difference being the additional functions (0-order differential operators) 
in the definitions of $\hat E_1$ and $\hat F_1$. In the classical case $q=1$, the operator $D_1$ can be obtained 
from the Dirac operator $D$ in Theorem \ref{Thm1} by the ``gauge transformation'' 
$D_1= \sqrt{y}\hs D  \sqrt{y}^{-1}$. 

On the other hand, if one rescales the 
volume form to $\mathrm{vol}_1 := \frac{1}{y}\hs \mathrm{vol}$ with a non-constant 
function $y$ without changing the Riemannian metric, 
then the  Dirac operator ceases to be self-adjoint 
but the above gauge transformed Dirac operator will remedy the problem. 
To see this, let $f,g, \sqrt{y}^{-1}\hsp f,\sqrt{y}^{-1}\hsp g \in \dom(D)$. Then
\begin{align*}
&\ip{\sqrt{y}\hs D  \sqrt{y}^{-1}\hsp f}{g}_{L_2(S,\frac{1}{y} \mathrm{vol})}  
= \int \ip{\sqrt{y}\hs D  \sqrt{y}^{-1}\hsp f}{g}\frac{1}{y}\hs\dd\mathrm{vol}
= \int \ip{ D  \sqrt{y}^{-1}\hsp f}{ \sqrt{y}^{-1}\hsp g} \dd\mathrm{vol}\\
&\quad= \int \ip{  f}{  \sqrt{y} D\sqrt{y}^{-1\hsp }g} \frac{1}{y}\dd\mathrm{vol}
= \ip{f}{\sqrt{y}\hs D  \sqrt{y}^{-1}\hsp  g}_{L_2(S,\frac{1}{y} \mathrm{vol})}. 
\end{align*}
For this reason and in view of \eqref{measure}, we may regard $D_1$ as the Dirac operator 
obtained from $D$ by rescaling the volume form 
$y \hs \dd_q y\hs \dd \phi\,\mapsto\,\dd_q y\hs \dd \phi$.

\section*{Acknowledgements}
This work is part of the project QUANTUM DYNAMICS and was partially supported by the grant H2020-MSCA-RISE-2015-691246, 
the Polish Government grant 3542/H2020/2016/2 and CIC-UMSNH.

\end{document}